\documentclass[11pt]{article}
\usepackage{amssymb,amsmath,euscript}
\newtheorem{proposition}{Proposition}[section]

\newtheorem{lemma}[proposition]{Lemma}
\newtheorem{theorem}[proposition]{Theorem}

\def\l{{\langle}}
\def\r{\rangle}
\def\dim{{\rm dim}_{_{\rm H}}}
\def\dimh{{\rm dim}_{_{\rm H}}}
\def\dimp{{\rm dim}_{_{\rm P}}}

\def\dimr{{\rm dim}^\rho_{_{\rm H}}}

\def\R{{\mathbb R}}

\def\Q{{\mathbb Q}}
\def\a{\alpha}
\def\la{\lambda}

\def\de{\delta}
\def\ga{\gamma}

\def\ep{\varepsilon}
\def\eps{\varepsilon}
\def\si{\sigma}

\def\Re {{\rm Re}\,}

\def\E{{\mathbb E}}
\def\P{{\mathbb P}}
\def\Z{{\mathbb Z}}
\def\N{{\mathbb N}}

\makeatletter \@addtoreset{equation}{section} \makeatother
\newcommand {\qed}%
{%
    {}\hfill
    {}\hfill
    {$\square $}%
    \vspace {0.3cm}%
    \pagebreak [2]%
    \par
}%
\newenvironment{proof}[1]{%
    \vspace{0.3cm}%
    \pagebreak [2]%
    \par%
    \noindent {\bf  Proof~#1\ }}{\qed}%
\newenvironment{remark}{%
    \vspace{0.3cm} \pagebreak [2]%
    \par%
    \refstepcounter{proposition}
    \noindent%
   {\bf Remark~\theproposition\  }}{\ }%
\def\Re {{\rm Re}\,}

\def\E{{\mathbb E}}
\def\P{{\mathbb P}}
\def\Z{{\mathbb Z}}
\def\N{{\mathbb N}}

\def\cqfd{$\square$}

\begin{document}
%


\title{Wavelet Series Representation and Geometric Properties
of Harmonizable Fractional Stable Sheets}

\author{Antoine Ayache\\UMR CNRS 8524, Laboratoire Paul Painlev\'e, B\^at. M2\\
  Universit\'e de Lille \\ 59655 Villeneuve d'Ascq Cedex, France\\
E-mail: \texttt{Antoine.Ayache@univ-lille.fr}\\
\ \\
\and
Narn-Rueih Shieh\\Emeritus Room, 4F Astro-Math Building \\
National Taiwan University\\
Taipei 10617, Taiwan\\
E-mail: \texttt{shiehnr@ntu.edu.tw}\\
\and
Yimin Xiao 
\\Department of Statistics and Probability
\\ Michigan State University  \\
East Lansing, MI 48824, U.S.A.\\
E-mail: \texttt{xiao@stt.msu.edu}\\}

\maketitle

\begin{abstract}
Let $Z^H= \{Z^H(t), t \in \R^N\}$ be a real-valued $N$-parameter
harmonizable fractional stable sheet with index $H = (H_1,
\ldots, H_N) \in (0, 1)^N$.  We establish a random wavelet series
expansion for $Z^H$ which is almost surely convergent in all the
H\"older spaces $C^\gamma ([-M,M]^N)$, where $M>0$ and
$\gamma\in (0, \min\{H_1,\ldots, H_N\})$ are arbitrary. One of the
main ingredients for proving the latter result is the LePage
representation for a rotationally invariant stable random measure.

Also, let $X=\{X(t), t \in \R^N\}$ be an $\R^d$-valued harmonizable
fractional stable sheet whose components are independent copies of
$Z^H$. By making essential use of the regularity of its local times,
we prove that, on an event of positive probability, the formula for the 
Hausdorff dimension of the inverse image $X^{-1}(F)$ holds for all 
Borel sets $F \subseteq \R^d$. This is referred to as a uniform 
Hausdorff dimension result for the inverse images.
\end{abstract}

\medskip

      {Running head}:  Harmonizable Fractional Stable Sheets  \\

      {\it 2000 AMS Classification numbers}: 60G52; 60G17; 60G18; 60G60.\\



\section{Introduction}
\label{Sec:Int}

For any given $0 < \alpha < 2$ and $H = (H_1,$ $\ldots, H_N)\in
(0, 1)^N$, let $Z^{H} = \{Z^{H}(t), t \in \R^N\}$ be a real-valued
harmonizable fractional $\a$-stable sheet (HFSS, for brevity) with
index $H$, defined by
\begin{equation}\label{eq:HFSS-1}
Z^{H} (t) :=    \Re \int_{\R^N} \prod_{j=1}^N\,
\frac{e^{it_j\lambda_j}-1}{|\lambda_j|^{H_j+\frac{1}{\a}}} \,
\widetilde{M}_\a(d\lambda),
\end{equation}
where $\widetilde{M}_\alpha$ is  a complex-valued rotationally
invariant $\a$-stable random measure with Lebesgue control measure;
we refer to Chapter 6 of Samorodnitsky and Taqqu (1994), for a
detailed presentation of such random measures as well as the
corresponding stochastic integrals.

From (\ref{eq:HFSS-1}) it follows that $Z^{H}$ has the following
operator-scaling property: for any $N\times N$ diagonal matrix $E
= (b_{ij})$ with $b_{ii} = b_i > 0$ for all $1 \le i \le N$ and
$b_{ij}=0$ if $i\ne j$, we have
\begin{equation}\label{Eq:OSS}
\big\{ Z^{H}(E t),\, t \in \R^N \big\} \stackrel{d}{=} \bigg\{
\bigg(\prod_{j=1}^N b_j^{H_j}\bigg)\,   Z^{H}(t),\ t \in \R^N
\bigg\}.
\end{equation}
The property (\ref{Eq:OSS}) is also called ``multi-self-similarity" by
Genton, Perrin and Taqqu (2007). By using (\ref{eq:HFSS-1}), one can
verify that, along each canonical direction of $\R^N$, $Z^{H}$ becomes
a real-valued harmonizable fractional stable motion; a detailed
presentation of the latter process, as well as other self-similar
stable processes is given in Samorodnitsky and Taqqu (1994). When
the indices $H_1,\ldots, H_N$ are not the same, $Z^{H}$ has different
scaling behavior along different directions of $\R^N$. Namely, $Z^{H}$ is
anisotropic in the ``time'' variable.

If one replaces in (\ref{eq:HFSS-1}) the parameter $\alpha$ by $2$, then
$Z^H$ becomes a fractional Brownian sheet (FBS, for brevity) denoted by $B^H$.
Sample path properties of FBS and several related classes of anisotropic Gaussian
random fields have been studied by several authors; see, for example, Kamont (1996),
Ayache {\it et al.} (2002), Ayache and Xiao (2005), Ayache {\it et al.} (2008),
Wu and Xiao (2007, 2011), Xiao (2009) and the references therein.
Observe that, up to a multiplicative constant, the Gaussian field $B^H$, can
also be represented as a moving average Wiener integral, in which the complex-valued
kernel $\prod_{j=1}^N\,\big(e^{it_j\lambda_j}-1\big)|\lambda_j|^{-H_j-1/2}$
is replaced by the real-valued kernel $\prod_{j=1}^N\,\big\{(t_j-s_j)_{+}^{H_j-1/2}
-(-s_j)_{+}^{H_j-1/2}\big\}$, with the convention that for all real-numbers $x$ and
$\beta$, $(x)_{+}^\beta=x^\beta$ when $x>0$ and $(x)_{+}^\beta=0$ else.
However, for $0 < \alpha < 2$, the harmonizable $\a$-stable field $Z^H$ does not
have such a property, and it is very different from the real-valued linear fractional
stable field (LFSS, for brevity) $Y^{H} = \{Y^{H}(t), t \in \R^N\}$ which is defined by
\begin{equation}
\label{Eq:LFSS}
Y^H (t):=\int_{\R^N} \prod_{j=1}^N\,\Big\{(t_j-s_j)_{+}^{H_j-\frac{1}{\a}}
-(-s_j)_{+}^{H_j-\frac{1}{\a}}\Big\}\,M_\a (ds),
\end{equation}
where $M_\a$ is a real-valued $\a$-stable random measure.

The main purpose of the present article is to study sample path
properties of HFSS and compare them with the properties of LFSS
obtained in Ayache {\it et al.} (2009).
In Section~\ref{sec:wav-HFSS} we establish a  random wavelet series
representation for $Z^H$; to this end, we expand the corresponding
harmonizable kernel in (\ref{eq:HFSS-1}) in terms of the Fourier
transform of the tensor product of a Lemari\'e-Meyer wavelet basis
for $L^2(\R)$. One of the main difficulties in this matter is that,
when $\a < 2$ we do not know whether the Fourier transform of a
Lemari\'e-Meyer wavelet basis is an unconditional basis for $L^\a (\R)$.
This is in contrast with the case of
LFSS in Ayache {\it et al.} (2009), where one does not have to use the
Fourier transform of a wavelet basis. 
In order to overcome this difficulty, we use some
specific properties of the harmonizable kernel as well as the fact
that the Fourier transform of a Lemari\'e-Meyer mother wavelet is
compactly supported and vanishes in a neighborhood of the origin.
Also, it is worth noticing that, the behavior of the random
coefficients, in our wavelet expansion of HFSS, can be estimated
through their LePage representations. This,  in turn, leads to the
conclusion that the wavelet expansion of $Z^H$ is, for any $\a\in
(0,2)$, almost surely convergent in all the H\"older spaces $C^\gamma
([-M,M]^N)$, where $M>0$ and $\gamma \in (0, \min\{H_1,\ldots, H_N\})$
are arbitrary constants.  In the case of LFSS, a weaker result holds:
under the condition that $\min\{H_1,\ldots, H_N\}>1/\a$, the wavelet
expansion of $Y^H$ is almost surely convergent in all the H\"older
spaces $C^\gamma ([-M,M]^N)$, where $M>0$ and $\gamma<\min\{H_1,
\ldots, H_N\}-1/\a$  (see Ayache {\it et al.} (2009)). Notice that, in contrast
with HFSS, when $\min\{H_1,\ldots, H_N\}\le 1/\a$, the LFSS sample path
is discontinuous and even becomes unbounded on every open set when 
the latter inequality is strict. The results in this section may be
compared with those of Ayache and Boutard (2017), where regularity
properties of harmonizable stable random fields with stationary increments
are studied by using the wavelet method, and those of Xiao (2010), Bierm\'e
and Lacaux  (2009, 2015), Panigrahi, Roy and Xiao (2018), where uniform
modulus of continuity of stable random fields including harmonizable ones
are studied by using different methods.

In Section \ref{sec:Geo}, we consider the $(N, d)$ harmonizable fractional 
stable sheet $X = \{X(t), t \in \R^N\}$ with values in $\R^d$ defined by
\begin{equation}\label{def:X}
X(t) = \big(X_1(t), \ldots, X_d(t)\big),
\end{equation}
where $X_1, \ldots, X_d$ are independent copies of $Z^{H}$. 
Our main result of this section is Theorem \ref{Th:IMdim2} which establishes
a uniform Hausdorff dimension result for the inverse images $X^{-1}(F)$ 
for all Borel sets $F \subset \R^d$, which are either deterministic or random.
This theorem is new even for fractional Brownian sheets (i.e., $\alpha = 2$)
and solves the problem raised in Remark 2.8 in Bierm\'e, Lacaux and Xiao (2009). 
The methodology for proving this uniform Hausdorff dimension result is 
based on the regularity properties (e.g., H\"older conditions in the set-variable)
of the local times of $X$ which is of independent interest for $(N, d)$
random fields.  The results of this section motivate several questions that will 
need further investigation. At the end of Section 3, we provide some remarks 
and open questions.

Throughout this paper we use $|\cdot|$ to
denote the Euclidean norm in $\R^N$. The inner product and Lebesgue
measure in $\R^N$ are denoted by $\l \cdot, \cdot\r$ and $\la_N$,
respectively. A vector $t\in\R^N$ is written as $t =
(t_1, \ldots, t_N)$, or as $\langle c\rangle $ if $t_1= \cdots =t_N
= c$. For any $s, t \in \R^N$ we write $s \prec t$ if $s_j < t_j$
for all $j = 1, \ldots, N$. In this case $[s, t] = \prod^N_{j=1}\, [s_j,
t_j]$ is called a closed interval (or a rectangle). We will let
${\cal A}$ denote the class of all closed intervals in $\R^N$. For
two functions $f$ and $g$, the notation $f(t) \asymp g (t)$ for $t
\in T$ means that the function $f(t)/g(t)$ is bounded from below and
above by positive constants that do not depend on $t \in T$.

We will use $c$ and $c(n)$ to denote unspecified positive and finite
constants, the latter depends on $n$. Both of them may not be
the same in each occurrence. More specific constants 
are numbered as $ c_{1}, c_{2}, \ldots$.

\vspace{.2in}

{\bf Acknowledgements.}\ \  This work was finished during
Yimin Xiao's visit to Universit\'e de Lille sponsored by CEMPI 
(ANR-11-LABX-0007-01). The hospitality from this university and 
the financial support from CEMPI are appreciated. The research 
of Yimin Xiao is partially supported by NSF grants
DMS-1612885 and DMS-1607089.

\section{Wavelet series representation of $Z^H$}
\label{sec:wav-HFSS}

The goal of this section is to establish a random wavelet series
representation for HFSS $Z^H$. First we fix some notation that
will be extensively used in the sequel.
\begin{itemize}
\item[(i)] The function $\psi$ denotes a usual Lemari\'e-Meyer
mother wavelet, see Lemari\'e and Meyer (1986), Meyer (1992)
and Daubechies (1992). The function $\psi$ satisfies the following
properties:

$(a)$ $\psi$ is  real-valued and belongs to the Schwartz class
$S(\R)$.

$(b)$ $\widehat{\psi}$, the Fourier transform of $\psi$, is
compactly supported and its support is contained in the ring
$\{\xi\in\R\,:\,\frac{2\pi}{3}\le |\xi|\le \frac{8\pi}{3}\}$.
Recall
that the Fourier transform of a function $f\in L^2 (\R)$ is the
limit of the Fourier transforms of functions of the Schwartz class
$S(\R)$ converging to $f$; throughout this article the Fourier
transform over $S(\R)$ is defined as $({\cal
F}g)(\xi)=\widehat{g}(\xi) = (2\pi)^{-1/2}\int_{\R}
e^{-is\cdot\xi} g(s)\,ds$ and the inverse map as $({\cal
F}^{-1}h)(s)=(2\pi)^{-1/2} \int_{\R} e^{is\cdot\xi} h(\xi)\,d\xi$,
thus the Fourier transform is a bijective isometry from $L^2
(\R)$ to itself.

$(c)$ The sequence $\{\psi_{j,k}\,:\,(j,k)\in\Z^2\}$ forms an orthonormal
basis of $L^2(\R)$, where
\begin{equation}
\label{eq:def-psijk}
 \psi_{j,k}(x):=2^{j/2}\psi(2^{j}x-k), \quad \forall\ x\in\R.
\end{equation}
 Note that the isometry
property of the Fourier transform implies that the sequence
$\Big\{\overline{\widehat{\psi}_{j,k}}\,:\,(j,k)\in\Z^2\Big\}$
forms an orthonormal basis of $L^2(\R)$ as well. Moreover, simple
computation  shows that, for all $\xi\in\R$,
\begin{equation}
\label{eq:exp-hatpsijk}
\widehat{\psi}_{j,k}(\xi)=2^{-j/2}e^{-ik2^{-j}\xi}\,\widehat{\psi}(2^{-j}\xi).
\end{equation}
For the sake of convenience, for each $v\in (0,1)$ and $x\in\R$,
we set
\begin{equation}
\label{eq:coffL2F} w_{j,k}^v (x):=\int_{\R}f_v
(x,\xi)\widehat{\psi}_{j,k}(\xi)\,d\xi,
\end{equation}
where the function $f_v(x,\cdot)$ is defined by $f_v(x,0)=0$ and, for all
$\xi\in\R\setminus\{0\}$, by
\begin{equation}
\label{eq:deffonfv}
f_v(x,\xi):=\frac{e^{ix\xi}-1}{|\xi|^{v+\frac{1}{\a}}}.
\end{equation}
Observe that when $f_v(x,\cdot)\in L^2(\R)$ (this is equivalent to $1 <
2v + \frac 2 \alpha < 3$), then, by using the fact that
$\Big\{\overline{\widehat{\psi}_{j,k}}\,:\,(j,k)\in\Z^2\Big\}$ is an
orthonormal basis of the latter space, one has
\begin{equation}
\label{eq:decompFL2} f_v(x,\cdot)=\sum_{j,k\in\Z}w_{j,k}^v
(x)\overline{\widehat{\psi}_{j,k}(\cdot)},
\end{equation}
where the series converges in the $L^2(\R)$ norm.
Roughly speaking, the key idea for obtaining a
wavelet representation for HFSS consists in showing that the
equality (\ref{eq:decompFL2}) holds in $L^\a(\R)$, even
when $f_v(x,\cdot)\notin  L^2(\R)$. For this purpose,  it is convenient
to renormalize the functions $\widehat{\psi}_{j,k}$ in such a way that their
$L^\a(\R)$ norms be equal to $\|\widehat{\psi}\|_{L^\a(\R)}$.

\item[(ii)] For each $(j,k)\in\Z^2$ and $\xi\in\R$, we set
\begin{equation}
\label{eq:Lapsijk}
\widehat{\psi}_{\a,j,k}(\xi):=2^{-j(\frac{1}{\a}-\frac{1}{2})}
\widehat{\psi}_{j,k}(\xi)
=2^{-j/\a}e^{-ik2^{-j}\xi}\,\widehat{\psi}(2^{-j}\xi).
\end{equation}
Moreover for every $x\in\R$, we set
\begin{equation}
\label{eq:Lawjk}
w_{\a,j,k}^v(x):=2^{j(\frac{1}{\a}-\frac{1}{2})}w_{j,k}^v (x)
=2^{j(\frac{1}{\a}-1)}\int_{\R}\frac{e^{ix\xi}-1}{|\xi|^{v+1/\a}}e^{-ik2^{-j}\xi}\,
\widehat{\psi}(2^{-j}\xi)\,d\xi.
\end{equation}
Observe that, for all $x, \, \xi \in\R$, one has
\begin{equation}
\label{eq:LaL2}
w_{\a,j,k}^v(x)\overline{\widehat{\psi}_{\a,j,k}(\xi)}=w_{j,k}^v(x)
\overline{\widehat{\psi}_{j,k}(\xi)}.
\end{equation}

\item[(iii)] In order to conveniently express $w_{\a,j,k}^v(x)$,
let us introduce, for each  fixed $v\in (0,1)$, the function $\psi^{v}$,
defined  as
\begin{equation}
\label{eq:fracprimpsi} \psi^{v}(y):=\int_{\R}e^{iy\eta}
\frac{\widehat{\psi}(\eta)}{|\eta|^{v+1/\alpha}}\;d\eta,\quad \forall\ y\in\R.
\end{equation}
Observe that in view of the properties $(a)$ and $(b)$ of the
Lemari\'e-Meyer mother wavelet $\psi$, one can easily show that
$\psi^{v}$ is a well-defined real-valued function which belongs to
$S(\R)$. Furthermore, the change of variable $\eta=2^{-j}\xi$
in (\ref{eq:Lawjk}) yields that, for all $x\in\R$,
\begin{equation}
\label{eq:expLawjk} w_{\a,j,k}^v(x)=2^{-jv}\big\{\psi^{v}(2^j
x-k)-\psi^{v}(-k)\big\}.
\end{equation}
\item[(iv)] Let $\{\epsilon_{J,K},\, (J,K)\in\Z^N\times\Z^N\}$ be the
sequence of complex-valued S$\a$S random variables defined as
\begin{equation}
\label{eq:def-epsilon}
\epsilon_{J,K}=\int_{\R^N}\overline{\widehat{\Psi}_{\a,J,K}(\lambda)}\,
\widetilde{M}_{\a}(d\lambda),
\end{equation}
where
\begin{equation}
\label{eq:bigpsijk}
\widehat{\Psi}_{\a,J,K}(\lambda):=\prod_{l=1}^N
\widehat{\psi}_{\a,j_l,k_l}(\lambda_l), \quad \forall \lambda \in \R^N.
\end{equation}
It is easy to verify that, for every $(J,K)\in\Z^N\times\Z^N$, the
scale parameter of $\epsilon_{J,K}$ is
$
\|\epsilon_{J,K}\|_\a=\|\widehat{\psi}\|_{L^\a(\R)}^N.
$
Hence the random variables $\{\epsilon_{J,K},\, (J,K)\in\Z^N\times\Z^N\}$
are identically distributed.
\end{itemize}

We are now in position to state our first main result of the section.

\begin{proposition}
\label{thm:wav-prob} Let $({\mathcal D}_n)_{n\in\N}$ be an increasing 
sequence of finite subsets of $\Z^N\times\Z^N$ which ``converges" to 
$\Z^N\times\Z^N$ (i.e., for every $n\in \N$, ${\mathcal D}_n\subset 
{\mathcal D}_{n+1}$ and $\cup_{n\in\N} {\mathcal D}_n=\Z^N\times\Z^N$).
For each $n\in\N$ and $t\in\R^N$, we denote by $U_n(t)$ the real-valued
random variable defined as
\begin{equation} \label{eq:def-Un}
U_{n}(t)=\Re\sum_{(J,K)\in {\mathcal D}_n}2^{-
\langle J,H \rangle }\epsilon_{J,K}\prod_{l=1}^N
\left\{\psi^{H_l}(2^{j_l}t_l-k_l)-\psi^{H_l}(-k_l)\right\}.
\end{equation}
Then $U_{n}(t)$ converges in probability to $Z^H(t)$ when $n$ goes
to infinity.
\end{proposition}

In order to be able to prove Proposition~\ref{thm:wav-prob}, we
need some preliminary results.
\begin{remark}
\label{rem:metLa} Recall that, in contrast to the case where $\a
\ge 1$, when $\a\in (0,1)$, the map $f\mapsto \Big
(\int_{\R^N}|f(\lambda)|^\a \;d\lambda\Big)^{1/\a}$ is no longer a
norm on $L^\a (\R^N)$; however one can define 
a metric $\Delta$ on this space by
\begin{equation}
\label{eq1:metric-delta} \Delta_ {L^\a (\R^N)}(f,g):=\int_{\R^N}
|f(\lambda)-g(\lambda)|^\a\;d\lambda, \quad \forall \ (f,g)\in 
L^\a (\R^N)\times L^\a (\R^N).
\end{equation}
Moreover the resulting metric space is complete. For the sake of convenience,
when $\a \ge 1$, one sets for every $(f,g)\in L^\a (\R^N)\times L^\a (\R^N)$,
\begin{equation}
\label{eq2:metric-delta} \Delta_ {L^\a
(\R^N)}(f,g):=\|f-g\|_{L^\a (\R^N)}=
\left(\int_{\R^N}|f(\lambda)-g(\lambda)|^\a\;d\lambda\right)^{1/\a}.
\end{equation}
For simplicity we will abuse the notation slightly and write $\|f\|_{L^\a (\R^N)}$ as
$\|f \|_\a$, which should not be confused with the scale parameter of a S$\a$S random
variable.
\end{remark}

We will make use of the following elementary lemma.
\begin{lemma} \label{rem:metabsconv}
Let $(E,d)$ be a complete metric vector space such that the metric $d$ is
translation invariant, namely for
all $x,y,z\in E$ one has,
\begin{equation}
\label{eq1:metabsconv} d(x+z,y+z)=d(x,y).
\end{equation}
Let $(a_i)_{i\in {\mathcal I}}$ be an arbitrary sequence
of elements of $E$ which satisfies
\begin{equation}
\label{eq2:metabsconv} \sum_{i\in {\mathcal I}}d(a_i,0)<+\infty.
\end{equation}
Then there is a unique element $a\in E$ satisfying the following 
property: For each increasing sequence $({\mathcal D}_n)_{n\in\N}$ of finite
subsets of ${\mathcal I}$ which converges to ${\mathcal I}$ (i.e.,
 ${\mathcal D}_n\subset {\mathcal D}_{n+1}$ for every $n\in \N$
and $\cup_{n\in\N} {\mathcal D}_n={\mathcal I}$), one has
$$
\lim_{n\rightarrow +\infty} d\Big(a,\sum_{i\in{\mathcal D}_n}
a_i\Big)=0.
$$
\end{lemma}

\noindent {\bf Proof} \; Relation (\ref{eq1:metabsconv}) implies that, for every
$p\in\N$ and $m\in\N$,
\begin{equation}
\label{eq3:metabsconv}
d\Big(\sum_{i\in{\mathcal
D}_{m+p}}a_i,\sum_{i\in{\mathcal D}_{m}}a_i\Big)=
d\Big(\sum_{i\in {\mathcal D}_{m+p}\setminus {\mathcal
D}_{m}}a_i,0\Big).
\end{equation}
Moreover, using the triangle inequality and (\ref{eq1:metabsconv}), one can prove,
by induction on the cardinality of ${\mathcal D}_{m+p}\setminus {\mathcal D}_{m}$, that,
\begin{equation}
\label{eq4:metabsconv}
d\Big(\sum_{i\in {\mathcal D}_{m+p}\setminus {\mathcal D}_{m}}a_i,0\Big)\le
\sum_{i\in {\mathcal I}\setminus {\mathcal
D}_m} d(a_i,0).
\end{equation}
Putting together (\ref{eq2:metabsconv}), (\ref{eq3:metabsconv}) and
(\ref{eq4:metabsconv}), one can easily show that
$\big(\sum_{i\in{\mathcal D}_n} a_i\big)_{n\in\N}$ is a Cauchy
sequence of $(E,d)$, which in turn implies that it converges to
some limit $a\in E$. Let us now show that the limit $a$ does not
depend on the choice of the sequence $({\mathcal D}_n)_{n\in\N}$.
Let $({\mathcal D'}_n)_{n\in\N}$ and $({\mathcal D''}_n)_{n\in\N}$
be two arbitrary increasing sequences of finite subsets of
${\mathcal I}$ which converge to ${\mathcal I}$. We denote by
$a'$ and $a''$, respectively, the limits of the sequences
$(\sum_{i\in{\mathcal D'}_n} a_i)_{n\in\N}$ and
$(\sum_{i\in{\mathcal D''}_n} a_i)_{n\in\N}$.
Observe that there exists an increasing map $\phi:\N\rightarrow\N$
such that, for every $n\in\N$, one has ${\mathcal D}'_n\subseteq
{\mathcal D''}_{\phi(n)}$. Therefore, it follows from
(\ref{eq1:metabsconv}) that, for each $n\in\N$,
\begin{equation}
\label{eq5:metabsconv}
d\Big(\sum_{i\in{\mathcal D''}_{\phi(n)}} a_i, \sum_{i\in{\mathcal
D'}_n} a_i\Big) = d\Big(\sum_{i\in{\mathcal
D''}_{\phi(n)}\setminus {\mathcal D'}_n} a_i, 0\Big)\le\sum_{i\in
{\mathcal I}\setminus {\mathcal D'}_n} d(a_i,0),
\end{equation}
where the last inequality is derived in the same way as the
inequality (\ref{eq4:metabsconv}). Finally, letting $n$ go
to infinity, one obtains, in view of (\ref{eq2:metabsconv})
and (\ref{eq5:metabsconv}), that $d(a'',a')=0$. \cqfd

\medskip 


In the following lemma, recall that $w_{\a,j_l,k_l}^{H_l}(t_l)$ is  defined in
(\ref{eq:Lawjk}) and $\widehat{\Psi}_{\a,J,K}$ in (\ref{eq:bigpsijk}).

\begin{lemma}
\label{lem:absconvKHFSS} For each fixed $t\in\R^N$, one has:
\begin{equation}
\label{eq1:absconvKHFSS} \sum_{(J,K)\in\Z^N\times\Z^N}\Delta_{L^\a
(\R^N)}\bigg (\Big\{\prod_{l=1}^N
w_{\a,j_l,k_l}^{H_l}(t_l)\Big\}\overline{\widehat{\Psi}_{\a,J,K}},\,
0\bigg)<+\infty.
\end{equation}
Consequently, in view of Lemma
\ref{rem:metabsconv}, there exists a function $F(t)\in L^\a (\R^N)$ such that,
for every increasing sequence $({\mathcal D}_n)_{n\in\N}$ of
finite subsets of $\Z^N\times\Z^N$ which converges to
$\Z^N\times\Z^N$, one has
\begin{equation}
\label{eq2:absconvKHFSS} \lim_{n\rightarrow +\infty}\Delta_{L^\a
(\R^N)}\bigg(\sum_{(J,K)\in {\mathcal D}_n}\Big\{\prod_{l=1}^N
w_{\a,j_l,k_l}^{H_l}(t_l)\Big\}\overline{\widehat{\Psi}_{\a,J,K}},\, F(t)\bigg)=0.
\end{equation}
\end{lemma}

In order to show (\ref{eq1:absconvKHFSS}), we first estimate the
decreasing rate of the coefficients $w_{\a,j,k}^v(x)$.
\begin{lemma}\label{rem:estimwjk}
For any constant $L > 0$, the following two results hold:
\begin{itemize}
\item[(i)] There exists a constant $c_1>0$,  depending only on $L$,
$v$ and $\a$, such that for each
$x\in\R$, $j\in\Z_+$ and $k\in\Z$, one has
\begin{equation}
\label{eq1:estimwjk} |w_{\a,j,k}^v(x)|\le c_1 2^{-jv}\Big\{\big(2+|2^j
x-k|\big)^{-L}+\big(2+|k|\big)^{-L}\Big\}.
\end{equation}
\item[(ii)]  For any positive number $M$,   there exists a constant $c_2>0$,
depending only  on $L$, $v$, $\a$ and $M$, such that for each
$x\in [-M,M]$, $j\in\Z_{-}$ and $k\in\Z$, one has
\begin{equation}
\label{eq2:estimwjk} |w_{\a,j,k}^v(x)|\le c_2 2^{(1-v)j}\big(2+|k|\big)^{-L}.
\end{equation}
\end{itemize}
\end{lemma}

\noindent {\bf Proof}\; Part $(i)$  follows easily from (\ref{eq:expLawjk})
and the fact that $\psi^v\in S(\R)$. Next we prove Part $(ii)$. Denote by
$(\psi^v)'$ the derivative of $\psi^v$, then by using again the fact that
$\psi^v\in S(\R)$, one has, for all $y\in\R$,
\begin{equation}
\label{eq3:estimwjk}
\big|(\psi^v)'(y)\big|\le c_3 \big (2+M+|y|\big)^{-L},
\end{equation}
where $c_3$ is a constant  depending on $L$, $v$, $\a$ and $M$ only.
On the other hand (\ref{eq:expLawjk}) and the Mean Value Theorem imply
that, for each $x\in [-M,M]$, $j\in\Z_{-}$ and $k\in\Z$, one has
\begin{equation}
\label{eq4:estimwjk}
w_{\a,j,k}^v(x)=2^{j(1-v)}x(\psi^v)'(2^j d -k),
\end{equation}
where $d\in (-M,M)$. Thus, combining (\ref{eq3:estimwjk}) with
(\ref{eq4:estimwjk}) and the triangle inequality, one gets
(\ref{eq2:estimwjk}).
\cqfd\\

\noindent {\bf Proof of Lemma~\ref{lem:absconvKHFSS}}\; We will
only give the proof for the  case where $\a\in (0,1)$,  the proof
in the case $\a\in [1,2)$ is rather similar. First, notice that,
for every $(J,K)\in\Z^N\times\Z^N$, one has,
$$
\int_{\R^N}\left|\overline{\widehat{\Psi}_{\a,J,K}(\lambda)}\right|^\a\,d\lambda
=\Big(\int_{\R}|\widehat{\psi}(\xi)|^\a\,d\xi\Big)^N:= c_4.
$$
Therefore, using (\ref{eq1:metric-delta}), one gets that, for all $t\in\R^N$,
\begin{equation}
\label{eq3:absconvKHFSS}
\sum_{(J,K)\in\Z^N\times\Z^N}\Delta_{L^\a
(\R^N)}\Big (\Big\{\prod_{l=1}^N
w_{\a,j_l,k_l}^{H_l}(t_l)\Big\}\overline{\widehat{\Psi}_{\a,J,K}},0\Big)
= c_4\prod_{l=1}^N \Big (\sum_{(j_l,k_l)\in \Z\times\Z}
|w_{\a,j_l,k_l}^{H_l}(t_l)|^\a\Big).
\end{equation}
On the other hand, it follows from  (\ref{eq1:estimwjk}), (\ref{eq2:estimwjk}), and
the inequality $(x+y)^\a\le x^\a+y^\a$ for  $x,\, y \ge 0$ that, for every $l=1,\ldots, N$,
\begin{equation}
\label{eq4:absconvKHFSS}
\sum_{(j_l,k_l)\in \Z\times\Z}
|w_{\a,j_l,k_l}^{H_l}(t_l)|^\a \le  c_5\, \sum_{j_l=0}^{+\infty}2^{-j_l\a \gamma}
\sum_{k_l=-\infty}^{+\infty}\Big\{(2+|2^{j_l} t_l-k_l|)^{-L\a}+(2+|k_l|)^{-L\a}\Big\},
\end{equation}
where $\gamma= \min\{H_l,1-H_l\}$ and $c_5$ is a constant independent of $(j_l,k_l)$. Moreover
for every constant  $L$ which satisfies $L\a>1$, we have
\begin{equation}
\label{eq5:absconvKHFSS}
\sup_{y\in\R}\bigg\{\sum_{k=-\infty}^{+\infty}(2+|y-k|)^{-L\a}\bigg\}<+\infty.
\end{equation}
Finally, putting together (\ref{eq3:absconvKHFSS}),
(\ref{eq4:absconvKHFSS}) and(\ref{eq5:absconvKHFSS}), one gets 
(\ref{eq1:absconvKHFSS}).
\cqfd
\\

The following lemma shows that the function $F(t)$  in (\ref{eq2:absconvKHFSS}) 
is in fact the kernel function corresponding to HFSS (see (\ref{eq:HFSS-1})).
\begin{lemma}
\label{lem:FeqKHFSS} For every $t\in\R^N$ and for Lebesgue almost
all $\lambda\in\R^N$, one has
\begin{equation}
\label{eq:FKHFSS} F(t,\lambda)=\prod_{l=1}^N
f_{H_l}(t_l,\lambda_l):=\prod_{l=1}^N
\frac{e^{it_l\lambda_l}-1}{|\lambda_l|^{H_l+\frac{1}{\a}}}.
\end{equation}
\end{lemma}

\noindent {\bf Proof} \, For any constant  $M >0$ and  $n\in\N$, we set
\begin{equation}
\label{eq:def-D-Dn}
D_{n,M}:=\{(j,k)\in\Z^2\,:\, |j|\le n \mbox{ and } |k|\le M2^{n+1}\}
\end{equation}
and
\begin{equation}
\label{eq:def-D-Dn-N}
D_{n,M}^N:=\big\{(J,K)\in \Z^N\times\Z^N\,:\, (j_l,k_l)\in D_{n,M}
\mbox{ for each $l=1,\ldots, N$}\big\}.
\end{equation}
Similarly to (\ref{eq1:absconvKHFSS}), one can show
that, for every fixed $t\in\R^N$ and  $l=1,\ldots, N$,
$$
\sum_{(j_l,k_l)\in\Z\times\Z}\Delta_{L^\a (\R)}\Big
(w_{\a,j_l,k_l}^{H_l}(t_l)\overline{\widehat{\psi}_{\a,j_l,k_l}},\, 0\Big)<+\infty.
$$
By applying Lemma \ref{rem:metabsconv} with  ${\mathcal D}_n=D_{n,M}$,
we see that  there exists a function $\widetilde{f_l}(t_l)\in L^\a (\R)$ such that
\begin{equation}
\label{eq1:convker1} \lim_{n\rightarrow +\infty}\Delta_{L^\a
(\R)}\Big(\sum_{(j_l,k_l)\in D_{n,M}}
w_{\a,j_l,k_l}^{H_l}(t_l)\overline{\widehat{\psi}_{\a,j_l,k_l}},\, 
\widetilde{f}_l(t_l)\Big)=0.
\end{equation}
It follows from (\ref{eq:bigpsijk}), (\ref{eq1:convker1}), (\ref{eq:def-D-Dn-N})
and (\ref{eq2:absconvKHFSS}) with  ${\mathcal D}_n=D_{n,M}^N$ that, for almost 
all $\lambda\in\R^N$,
$$
F(t,\lambda)=\prod_{l=1}^N \widetilde{f}_l(t_l,\lambda_l).
$$
Thus, in order to show (\ref{eq:FKHFSS}), it is
sufficient to prove that, for every $l=1,\ldots, N$ and for almost
all $\lambda_l\in\R$, one has
\begin{equation}
\label{eq:tildefeqf}
\widetilde{f}_l(t_l,\lambda_l)=f_{H_l}(t_l,\lambda_l).
\end{equation}
For every $m\in\N$, let $h_m$ be the function defined, for all
$\lambda_l\in\R$, as
\begin{equation} \label{eq:defhm}
h_m(\lambda_l)=\left\{
\begin{array}{ll}
0, \quad &\hbox{ if $|\lambda_l|\le \frac{2^{-m+1}\pi}{3}$},\\
1, &\hbox{ else.}
\end{array}
\right.
\end{equation}
Observe that (\ref{eq:deffonfv}) and (\ref{eq:defhm}) imply
that $f_{H_l}(t_l,\cdot)h_m(\cdot) \in L^2 (\R)$ and, since
$\big\{\overline{\widehat{\psi}_{j_l,k_l}}\,:\,(j_l,k_l)\in
\Z^2\big\}$ is an orthonormal basis of the latter space,
one gets
\begin{equation}
\label{eq1:fhmL2} \lim_{n\rightarrow+\infty}\int_\R \Big
|f_{H_l}(t_l,\lambda_l)h_m(\lambda_l)-\sum_{(j_l,k_l)\in
D_{n,M}}\widetilde{w}_{j_l,k_l}(t_l)
\overline{\widehat{\psi}_{j_l,k_l}(\lambda_l)}\Big|^2\,d\lambda_l=0,
\end{equation}
where
\begin{equation}
\label{eq:tildewjkfh} \widetilde{w}_{j_l,k_l}(t_l):=\int_\R
f_{H_l}(t_l,\lambda_l)h_m(\lambda_l)\widehat{\psi}_{j_l,k_l}(\lambda_l)\,d\lambda_l.
\end{equation}
On the other hand, the property $(b)$ of $\psi$ (given at
the beginning of this section) and (\ref{eq:exp-hatpsijk}) entails
that
\begin{equation}
\label{eq:supphatpsijk}
\mbox{supp}\,\widehat{\psi}_{j_l,k_l}\subseteq
\Big\{\lambda_l\in\R\,:\,\frac{2^{j_l+1}\pi}{3}\le |\lambda_l|\le
\frac{2^{j_l+3}\pi}{3}\Big\}.
\end{equation}
Putting together (\ref{eq:tildewjkfh}), (\ref{eq:defhm}),
(\ref{eq:supphatpsijk}) and (\ref{eq:coffL2F}),
we see that, for every $j_l\ge -m$ and $k_l\in\Z$,
\begin{equation}
\label{eq:tildewjkeqwjk}
\widetilde{w}_{j_l,k_l}(t_l)=w_{j_l,k_l}^{H_l}(t_l).
\end{equation}
Denote by ${\mathcal C}_m$ the ring,
\begin{equation}
\label{eq:defcrowncm} {\mathcal C}_m:=
\Big\{\lambda_l\in\R\,:\,\frac{2^{-m+3}\pi}{3}\le |\lambda_l|\le
\frac{2^{m+3}\pi}{3}\Big\}\,.
\end{equation}
From now on, we assume that $\lambda_l\in{\mathcal C}_m$. Notice that
(\ref{eq:supphatpsijk}) implies that, for all $j_l\le -m$ and
$k_l\in\Z$, one has that
$$\widehat{\psi}_{j_l,k_l}(\lambda_l)=0.$$
Therefore, in view of (\ref{eq:tildewjkeqwjk}), one obtains that
\begin{equation}
\label{eq:sumcrowncm} \sum_{(j_l,k_l)\in
D_{n,M}}\widetilde{w}_{j_l,k_l}(t_l)\overline{\widehat{\psi}_{j_l,k_l}(\lambda_l)}
=\sum_{(j_l,k_l)\in  D_{n,M}}w_{j_l,k_l}^{H_l}(t_l)\overline{\widehat{\psi}_{j_l,k_l}
(\lambda_l)}.
\end{equation}
On the other hand (\ref{eq:defhm}) entails that
\begin{equation}
\label{eq:hmcrowncm}
f_{H_l}(t_l,\lambda_l)h_m(\lambda_l)=f_{H_l}(t_l,\lambda_l) \ \
\hbox{ for } \lambda_l \in{\mathcal C}_m .
\end{equation}
It follows from  (\ref{eq:sumcrowncm}), (\ref{eq:hmcrowncm})
and (\ref{eq1:fhmL2})  that
$$
\lim_{n\rightarrow+\infty}\int_{{\mathcal C}_m} \Big
|f_{H_l}(t_l,\lambda_l)-\sum_{(j_l,k_l)\in
D_{n,M}}w_{j_l,k_l}^{H_l}(t_l)\overline{\widehat{\psi}_{j_l,k_l}
(\lambda_l)}\Big|^2\,d\lambda_l=0.
$$
Since  ${\mathcal C}_m$ is a bounded set, H\"older's inequality
and (\ref{eq:LaL2}) imply that
\begin{equation}
\label{eq:limfHcrowncm} \lim_{n\rightarrow+\infty}\int_{{\mathcal
C}_m} \Big |f_{H_l}(t_l,\lambda_l)-\sum_{(j_l,k_l)\in
D_{n,M}}w_{\a,j_l,k_l}^{H_l}(t_l)\overline{\widehat{\psi}_{\a,j_l,k_l}
(\lambda_l)}\Big|^\a\,d\lambda_l=0.
\end{equation}
On the other hand, (\ref{eq1:convker1}) entails that
\begin{equation}
\label{eq:limtildefcrowncm}
\lim_{n\rightarrow+\infty}\int_{{\mathcal C}_m} \Big
|\widetilde{f}_l(t_l,\lambda_l)-\sum_{(j_l,k_l)\in
D_{n,M}}w_{\a,j_l,k_l}^{H_l}(t_l)\overline{\widehat{\psi}_{\a,j_l,k_l}
(\lambda_l)}\Big|^\a\,d\lambda_l=0.
\end{equation}
It follows from (\ref{eq:limfHcrowncm}) and (\ref{eq:limtildefcrowncm})
that, for almost all $\lambda_l\in {\mathcal C}_m$, one has
$f_{H_l}(t_l,\lambda_l)=\widetilde{f}_l(t_l,\lambda_l)$. Finally,
by using the latter equality and the fact that
$\R\setminus\{0\}=\cup_{m\in\N}{\mathcal C}_m$, one gets
(\ref{eq:tildefeqf}). \cqfd\\

We are now in position to show Proposition~\ref{thm:wav-prob}.

\noindent {\bf Proof of Proposition~\ref{thm:wav-prob}}\, This
proposition is a straightforward consequence of Lemmas
\ref{lem:absconvKHFSS} and \ref{lem:FeqKHFSS} as well as of
the following standard result on integrals with respect to
stable measures: if $g_n$ converges to $g$ in $L^\a (\R^N)$ then
$\int_{\R^N}g_n(\lambda)\,\widetilde{M}_\a(d\lambda)$ converges to
$\int_{\R^N}g(\lambda)\,\widetilde{M}_\a(d\lambda)$ in probability
(see Samorodnitsky and Taqqu (1994)). \cqfd

\medskip
The second main result of this section is the following theorem.

\begin{theorem}
\label{thm:unif-conv-wav} For any constant $M >0$ and $n\in\N$ we
denote by $U_{n}^*$ the real-valued random function defined, for
every $t\in\R^N$, as
\begin{equation*}
\label{eq:def-Un2} U_{n}^*(t)=\Re\sum_{(J,K)\in D_{n,M}^N}2^{-
\langle J,H \rangle }\epsilon_{J,K}\prod_{l=1}^N
\left\{\psi^{H_l}(2^{j_l}t_l-k_l)-\psi^{H_l}(-k_l)\right\},
\end{equation*}
where $D_{n,M}^N$ is defined in (\ref{eq:def-D-Dn-N}). Then,
the following two results hold:
\begin{itemize}
\item[(i)] with probability $1$, $\{U_{n}^*\}_{n\in\N}$ is a
Cauchy sequence in the H\"older space $C^\gamma\big([-M,M]^N\big)$
of any order $\gamma<\min\{H_1,\ldots , H_N\}$, the limit is
denoted by $U^*$;
\item[(ii)] the random field $U^*=\{U^*(t)\}_{t\in [-M,M]^N}$
is a version of the HFSS $\{Z^H(t)\}_{t\in [-M,M]^N}$.
\end{itemize}
\end{theorem}

\noindent {\bf Proof}\, Part $(i)$ of the
theorem can be proved by using the wavelet method similar to that
in the proof of Proposition~6 in Ayache {\it et al} (2009). It also follows
from Proposition~\ref{prop:estim-randcoeff} below, which provides a sharp
``deterministic" upper bound for the random variables $|\epsilon_{J,K}|$.
Part $(ii)$ of the theorem is an easy consequence of Part $(i)$ and
of Proposition~\ref{thm:wav-prob} with ${\mathcal D}_n=D_{n,M}^N$.
\cqfd

\begin{proposition}
\label{prop:estim-randcoeff} There is an event $\Omega^*$ of probability $1$
such that for all fixed $\eta>0$ there exists a random variable $C>0$ 
(depending on $\Omega^*$, $\eta$ and $\a$) which satisfies the following
property:
\begin{itemize}
\item[(i)] If $\a\in (0,1)$, then  for every
$\omega\in\Omega^*$ and for all $(J,K)\in\Z^N\times\Z^N$,
\begin{equation}
\label{eq1:dub-espJK} |\epsilon_{J,K}(\omega)|\le
C(\omega)\prod_{l=1}^N (1+|j_l|)^{1/\alpha+\eta}\,.
\end{equation}
\item[(ii)] If $\a\in [1,2)$, then  for every
$\omega\in\Omega^*$ and for all $(J,K)\in\Z^N\times\Z^N$,
\begin{equation}
\label{eq2:dub-espJK} |\epsilon_{J,K}(\omega)|\le
C(\omega)\prod_{l=1}^N
(1+|j_l|)^{1/\alpha+\eta}\sqrt{\log(2+|j_l|)}\sqrt{\log(2+|k_l|)}\,.
\end{equation}
\end{itemize}
\end{proposition}

The proof of Proposition~\ref{prop:estim-randcoeff} mainly relies on a
LePage series representation of the complex-valued S$\a$S process
$\big\{\epsilon_{J,K}\,:\,(J,K)\in\Z^N\times\Z^N\big\}$. We skip it
since it is similar to that of Lemma~2.7 in
Ayache and Boutard (2017).

\bigskip
\section{Uniform Hausdorff dimension result for the inverse images } 
\label{sec:Geo}
 
Let  $X= \{X(t), t \in \R^N\}$ be an $(N, d)$ harmonizable fractional
stable sheet  defined in (\ref{def:X}). For any Borel set $F \subset \R^d$, 
we define the inverse image $X^{-1}(F)$ by
$$X^{-1}(F) = \{t \in (0, \infty)^N: X(t) \in F\}.$$
Notice that we have avoided the boundary of $\R_+^N$ on which 
$X(t) \equiv 0$ a.s. This causes little loss of generality.

When $\alpha = 2$ (i.e., $X$ is a fractional Brownian sheet in $\R^d$)
and $F$ is fixed, Bierm\'e, Lacaux and Xiao (2009, Theorem 2.3) proved
the following result on the Hausdorff dimension of $X^{-1}(F)$:
\begin{equation}\label{Eq:L-infty}
\big\|\dim X^{-1}(F) \big\|_{L^\infty(\P)} =
\min_{1 \le k \le N} \bigg\{ \sum_{j=1}^k
\frac{H_k} {H_j} + N-k - H_k \big(d- \dim F \big)\bigg\},
\end{equation}
where for any function $Y: \Omega \to \R_+$,
$\|Y\|_{L^\infty(\P)}$ is defined as
\[
\|Y\|_{L^\infty(\P)}=
\sup\big\{\theta: Y \ge \theta \ \hbox{ on an event $E$ with }\, \P(E) > 0\big\}.
\]
Observe that in \eqref{Eq:L-infty}, the probability $\P\{\dim X^{-1}(F) > 0\}$ depends on $F$.
In Remark 2.8 in Bierm\'e, Lacaux and Xiao (2009), they asked the following question: If 
$\sum_{j=1}^N  1 /H_j > d$, does there exist a single event $\Omega_1 \subseteq \Omega$ 
of positive probability such that on $\Omega_1$,
\[
\dim X^{-1}(F) = \min_{1 \le k \le N} \bigg\{ \sum_{j=1}^k
\frac{H_k} {H_j} + N-k - H_k \big(d- \dim F \big)\bigg\}
\]
holds for \emph{all} Borel sets $F \subseteq \R^d$?  This is referred to as a uniform 
Hausdorff dimension problem for the inverse images of $X$.

Our objective of this section is to solve this problem for the $(N, d)$ harmonizable fractional
$\a$-stable sheet $X$ with $ \a \in [1, 2]$.  Our main result  is Theorem \ref{Th:IMdim2} below. 



\subsection{Some preliminaries results}
\label{Sec:Pre}

Let us collect some known results on harmonizable
fractional stable sheets, which will be useful for later sections.

For any $n \ge 1$ and  $t^1, \ldots, t^n \in \R^N$, the characteristic
function of the joint distribution of $Z^H(t^1), \ldots, Z^H(t^n)$ is
given by
\begin{equation}\label{Eq:ChF1}
\E\exp\bigg(i \sum_{j=1}^n u_j Z^H(t^j)\bigg) = \exp\bigg(- \Big\|
\sum_{j=1}^n u_j F(t^j, \cdot)\Big\|_{\a}^\a\bigg),
\end{equation}
where $u_j \in \R$ ($1 \le j\le n$),  $F(t, \lambda)$ is the function in
(\ref{eq:FKHFSS}). Recall that for every  $f \in L^\a (\R^N)$, $\|f \|_\a$
denotes $\|f\|_{L^\a (\R^N)}$.

It follows from \eqref{Eq:ChF1} that the scale parameter of the S$\alpha$S
random variable $\sum_{j=1}^n u_j Z^H(t^j)$ is
\begin{equation}\label{Eq:ScalePa}
\bigg\|\sum_{j=1}^n u_j Z^H(t^j)\bigg\|_\a := \bigg\|
\sum_{j=1}^n u_j F(t^j, \cdot)\bigg\|_{\a}.
\end{equation}
This allows us to describe some probabilistic properties of $Z^H$ by
the analytic properties of the functions $F(t, \cdot)$ and the
geometric structures of the space $L^\a (\R^N)$.


Lemma \ref{Lem:SLND} is proved in Xiao (2011) for $\alpha \in [1, 2)$
and in Wu and Xiao (2007) for $\a = 2$. Part (i) shows that for all $0 < a < b$
and $s, t \in [a, b]^N$ the scale parameter of $Z^{H}(s) - Z^{H}(t)$ is comparable with
$\rho(s, t)$, which is the metric $\rho$ on $\R^N$ defined by
\begin{equation}\label{Def:metric}
\rho(s, t) = \sum_{j=1}^N |s_j - t_j|^{H_j},\qquad \forall\, s, t
\in \R^N.
\end{equation}
Part (ii) says that $Z^{H}$ has the property of sectorial local nondeterminism.

\begin{lemma}\label{Lem:SLND}
Suppose $\alpha \in [1, 2]$ and $0 < a < b$ are constants. Then  there exist  
constants $c_{6}\ge 1$ and $c_{7} > 0$, depending on $a, \, b, \, H$ and $N$ only,
such that  the following properties hold:
\begin{itemize}
\item[(i)]\ For  all $s, t \in [a, b]^N$.
 \begin{equation} \label{eq:scale}
c_{6}^{-1}\   \rho(s, t) \le \big\|Z^{H}(s) - Z^{H}(t)
\big\|_\a \le c_{6}\ \rho(s,t).
\end{equation}

\item[(ii)]\ For all positive integers
$n \geq 2$ and all $ t^1, \ldots, t^n$ $\in [a, b]^N$, we have
\begin{equation} \label{eq:slnd-HFSS}
\Big\|Z^{H}(t^n)\, \big|\,   Z^H(t^1), \ldots ,
 Z^H(t^{n-1}) \Big\|_\a \ge c_{7}\
         \sum_{j=1}^N \min_{0 \le k \le n-1}
         \big| t^n_j - t^k_j \big|^{H_j},
\end{equation}
where $\big\|Z^{H}(t^n)\, \big|\,   Z^H(t^1), \ldots ,
 Z^H(t^{n-1}) \big\|_{\a}$ is the $L^\a(\R^N)$-distance from $F(t^n, \cdot)$
 to the subspace generated by $F(t^j, \cdot)$ ($j = 0, 1, \ldots, t^{n-1}$).
\end{itemize}
\end{lemma}

\begin{remark}\label{Re:LND}
We believe that (\ref{eq:slnd-HFSS}) still holds if $\a \in (0, 1)$,
but we have not been able to prove this. The Fourier analytic method
in Ayache and Xiao (2016) works well for proving the property of local
nondeterminism for harmonizable fractional $\a$-stable fields with
stationary increments for all $\a \in (0,\,2)$. Unfortunately, the last
part of the proof of Theorem 2.1 in Ayache and Xiao (2016) breaks down
for harmonizable fractional stable sheets.
\end{remark}

In order to make use of the property of sectorial local
nondeterminism, we need the following useful lemma from Xiao
(2011). The case of $\a= 2$ is included for completeness and in this case
it can be shown that $c(n)$ does not depend on $n$.
\begin{lemma}\label{Lem:PG-th}
Assume $ \a \in [1,  2]$ and $0 < a < b$ are constants. For all integers 
$n \ge 2$ there exists a positive constant $c(n)$ such that for all $t^1,\ldots, t^n \in
[a, b]^N$ and $u_1, \ldots, u_n \in \R$,
\begin{equation}\label{Eq:PG1}
\begin{split}
\bigg\|\sum_{j=1}^n u_j Z^H(t^j)\bigg\|_\a &\ge c(n)
\bigg(|v_1| \big\|Z^H(t^1)\big\|_\a \\
&\qquad \qquad \qquad + \sum_{j=2}^n |v_{j}|\,
\big\|Z^H(t^j)\big|Z^H(t^1), \ldots, Z^H(t^{j-1})\big\|_\a\bigg).
\end{split}
\end{equation}
In the above,
\begin{equation}\label{Eq:Transf-1}
(v_{1}, \ldots, v_{n}) = (u_1,\ldots, u_n)\, A,
\end{equation}
where $A= (a_{i j})$ is an $n\times n$ lower triangle matrix (which
depends on $t^1, \ldots, t^n$) with $a_{i i} = 1$ for all $1 \le i
\le n$.
\end{lemma}

\begin{remark}
Roughly speaking, in (\ref{Eq:PG1}) we expand
$\big\|\sum_{j=1}^n u_j Z^H(t^j)\big\|_\a$ by repeatedly
``conditioning'' $Z^H(t^j)$, given $Z^H(t^1), \ldots,
Z^H(t^{j-1})$. Moreover, this ``conditioning'' can be done in an
arbitrary order of the random variables $Z^H(t^1), \ldots,
Z^H(t^n)$. This observation will be useful for studying regularity
properties of the local times of $Z^H$ below.
\end{remark}

The sample function of $Z^H$ is  continuous and has
the following uniform modulus of continuity, which was proved in Xiao
(2010). A similar result for operator-scaling stable random fields
with stationary increments was proved by Bierm\'e and Lacaux
(2009, 2015).

\begin{lemma}\label{Lem:unimod}
Let $Z^H=\{Z^H(t),\, t \in \R^N\}$ be a real-valued harmonizable
fractional stable sheet defined by (\ref{eq:HFSS-1})  Then for any
constants  $0 \le a < b$ and 
$\ep > 0$, one has
\begin{equation} \label{eq:mod-cont}
\lim_{h \to 0} \sup_{\stackrel{s, t\in [a, b]^N} {|s-t|\le h}}\frac{|{Z}^H(s)-
{Z}^H(t)|} {\rho(s, t)  \big|\log
\rho(s, t)\big|^{1/\a+\ep}}= 0,
 \qquad \hbox{ a.s.}
\end{equation}
\end{lemma}

\subsection{Local times}
\label{Sec:LT}

Many geometric properties of a random field are related to the existence and
analytic properties of its local times. We recall briefly some aspects
of the theory of local times. For further information we refer to Geman and
Horowitz (1980) and Dozzi (2002).

Let $X: \R^N \to \R^d$ be a (deterministic or stochastic) Borel vector field.
For any Borel set $T \subseteq \R^N$, the occupation measure of $X$
on $T$ is defined as
\[
\mu_{_T}(\cdot) = \la_N \big\{ t \in T: X(t) \in \cdot \big\},
\]
which is a Borel measure on $\R^d$. If $\mu_{_T}$ is almost surely absolutely
continuous with respect to the Lebesgue measure $\la_d$, then $X(t)$
is said to have a \emph{local time} on $T$. The local time,
$L(\cdot, T)$, is defined as the Radon--Nikod\'ym derivative of
$\mu_{_T}$ with respect to $\lambda_d$, i.e.,
\[
    L(x,T) = \frac{d\mu_{_T}} {d\lambda_d}(x),\qquad \forall x\in\R^d.
\]
In the above, $x$ is called the \emph{space variable}, and $T$ is
the \emph{time variable}. It is clear that if $X$ has  local times on $T$,
then for every Borel set $S \subseteq T$, $L(x, S)$ also exists.
It follows from Theorems 6.3 and 6.4 in Geman and Horowitz (1980)
that the local times of $X$ have a version, still denoted by $L(x,
T)$, such that it is a kernel in the following sense:
\begin{itemize}
\item[(i)]\ For each fixed $S \subseteq T$, the function $x
\mapsto L(x, S)$ is Borel measurable in $x\in \R^d$.
\item[(ii)]\
For every $x \in \R^d$, $L(x, \cdot)$ is a Borel measure on
${\EuScript B}(T)$, the family of Borel subsets of $T$.
\end{itemize}
Moreover, $L(x, T)$ satisfies the following \emph{occupation density
formula}: For every measurable function $f : \R^d \to \R_+$,
\begin{equation}\label{Eq:occupation}
\int_T f(X(t))\, d t = \int_{\R^d} f(x) L(x, T)\, dx.
\end{equation}

Suppose we fix an interval $T = \prod_{i=1}^N [a_i, a_i+h_i]$ in
$\mathcal{A}$. If  we can choose a version of the local time, still
denoted by $L\big(x, \prod_{i=1}^N [a_i, a_i+t_i]\big)$, such that
it is a  continuous function of $(x, t_1, \cdots, t_N)$ $\in$
$\R^d\times \prod_{i=1}^N[0,h_i]$, then $X$ is said to have a
\emph{jointly continuous local time} on $T$.

The following result follows from Lemmas \ref{Lem:SLND} and
\ref{Lem:PG-th} above and Theorem 4.2 in Xiao (2011).

\begin{lemma}\label{Lem:JC}
Let $X= \{X(t), t \in \R^N\}$ be an $(N, d)$ harmonizable fractional
stable sheet defined by (\ref{def:X}).  If $\a \in [1, 2]$ and
$\sum_{j=1}^N 1/{H_{j}} > d$, then for any bounded interval $T
\subseteq (0, \infty)^N$, $X$ has a jointly continuous local time on $T$
almost surely.
\end{lemma}

It is known from Adler (1981) that, when a local time is jointly
continuous, $L(x, \cdot)$ can be extended to be a finite Borel
measure supported on the level set
\begin{equation}\label{eq:inv}
X_T^{-1}(x) = \{ t \in T: X(t) = x\}.
\end{equation}
In order to use this fact for proving Theorem \ref{Th:IMdim2} below,
we also need to obtain the H\"older conditions for the local time
$L(x, \cdot)$ in the set variable.

For convenience we assume in the rest of this paper that
\begin{equation}
\label{Eq:H} 0 < H_1 \leq \dots \leq H_N < 1.
\end{equation}
Of course, there is no loss of generality in the arbitrary ordering
of $H_1,\dots, H_N$.

Assume that $d<\sum_{j=1}^N 1/{H_{j}}$ holds and let $\tau \in \{1,
\ldots, N\}$ be the integer such that
\begin{equation}\label{Eq:Hcon2}
\sum_{\ell=1}^{\tau-1} \frac{1}{H_{\ell}} \le d <\sum_{\ell=1}^\tau
\frac{1}{H_{\ell}}
\end{equation}
with the convention that $ \sum_{\ell=1}^{0} \frac{1}{H_{\ell}} :=
0$. We denote in the sequel
\begin{equation}\label{Eq:beta}
\beta_\tau = \sum_{\ell=1}^\tau \frac{H_\tau} {H_\ell} + N - \tau -
H_\tau d.
\end{equation}
As shown by Theorem 7.1 in Xiao (2009), for any $x \in \R^d$, we have
$\dim X^{-1}(x) = \beta_\tau$ with positive probability.

We will make use of the following uniform H\"older condition for
the local time $L(x, \cdot)$ in the set variable. We refer to Nolan
(1989), K\^ono and Shieh (1993), Xiao (2011) for earlier results
of this type for self-similar stable processes and harmonizable
fractional stable motion. The new feature of Proposition \ref{Th:Holder}
is that the uniform H\"older condition is given in terms of $(H_1,
\ldots, H_N)$ and the side-lengths of $I$, so it characterizes the
anisotropy of $X$, and the exponents of $r_\ell$ are almost optimal.

\begin{proposition}\label{Th:Holder}
Let $X= \{X(t), t \in \R^N\}$ be an $(N, d)$ harmonizable fractional
stable sheet defined by (\ref{def:X}) and let $T \subset (0, \infty)^N$ be 
an interval. Suppose $\a \in [1, 2]$ and $\tau \in \{1, \ldots, N\}$
is the integer so that (\ref{Eq:Hcon2}) holds. Let $L(x, \cdot)$ be
the jointly continuous local time of $X$ on $T$. Then, for any $\ep>0$,
almost surely there exist positive and finite constants $c_{8} =
c_{8}(\omega)$ and $r_0 = r_0(\omega)$ such that
\begin{equation}\label{Eq:Holder4}
\max_{x \in \R^d}L(x, I) \le c_{8}\,  \bigg( \prod_{\ell=1}^\tau r_\ell^{1 -
\frac {H_\ell d}{p_\ell}} \, \prod_{\ell= \tau+ 1}^N
r_\ell\bigg)^{1-\ep}
\end{equation}
for all intervals $I = [a, a + \l r_\ell \r] \subseteq T$ with $r_\ell
\le r_0.$ Here $p_\ell\ge 1$ ($1 \le \ell \le \tau$) are positive numbers
such that $\sum_{\ell=1}^\tau p_\ell^{-1} = 1$.
\end{proposition}

The proof of Proposition \ref{Th:Holder} relies on the moment estimates
in Lemma \ref {Lem:M1} below and a chaining argument in Ehm (1981) [see
also Geman and Horowitz (1980, pp 49-50)]. Since this proof is standard,
we omit it.

We remark that, unlike the Gaussian cases considered in Xiao (1997) and
Wu and Xiao (2011), we are not able to replace the exponent $1 - \eps$
in \eqref{Eq:Holder4} by $1$ together with a logarithmic factor. This
is due to the inability in taking full advantage of the sectorial local
nondeterminism in Part (ii) of Lemma \ref{Lem:SLND}. Instead, our proof
of Lemma \ref{Lem:M1} below makes use of Lemma \ref{Lem:PG-th}.

\begin{lemma}\label{Lem:M1}
Assume the conditions of Proposition \ref{Th:Holder} hold. Then for any
positive constants $p_\ell\ge 1$ ($1 \le \ell \le \tau$) such that
$\sum_{\ell=1}^\tau p_\ell^{-1} = 1$ the following statements hold.
\begin{itemize}
\item[(i)]\, For all
integers $n \ge 1$, there exists a positive and finite constant
$c_{9}= c_{9}(n)$ such that for all intervals $I = [a,
a + \l r_\ell \r] \subseteq T$ with side-lengths $r_\ell \in (0, 1)$
and all  $x \in \R^d$
\begin{equation}\label{Eq:m11}
\E \big[L(x, I)^n\big] \le  c_{9}\,\bigg( \prod_{\ell=1}^\tau
r_\ell^{1 - \frac {H_\ell d}{p_\ell}} \, \prod_{\ell= \tau+ 1}^N
r_\ell\bigg)^n.
\end{equation}
\item[(ii)]\, For all
even integers $n \ge 2$ and all $\gamma \in (0, 1)$ small enough,
there exists a positive and finite constant $c_{10}=
c_{10}(n)$ such that for all intervals $I = [a, a+ \l r_\ell
\r] \subseteq T$, $x,\, y \in \R^d$ with $|x - y|\le 1$,
\begin{equation}\label{Eq:m20}
\E \Big[\big(L(x, I) - L(y, I)\big)^n\Big]  \le c_{10}\,
|x-y|^{n \gamma}\, \bigg( \prod_{\ell=1}^\tau r_\ell^{1 - \frac
{H_\ell d}{p_\ell}} \, \prod_{\ell= \tau+ 1}^N r_\ell\bigg)^n.
\end{equation}
\end{itemize}
\end{lemma}

\begin{proof}\ This is essentially proved in Lemmas 4.3 and 4.7 in
Xiao (2011). The difference is that $I$ is assumed to be a cube in
Xiao (2011) and here it is an arbitrary interval, which is needed
for proving Theorem \ref{Th:IMdim2} below.

Again we start with the following identities about the moments
of the local time and its increments from  Geman and Horowitz (1980):
For all $x, y \in \R^d$ and all integers $n \ge 1$,
\begin{equation} \label{Eq:moment1}
\begin{split}
\E \big[ L(x, I)^n\big] &= (2 \pi)^{- nd} \int_{I^n} \int_{{
\R}^{nd}} \exp \bigg( - i \sum_{j=1}^{n} \l u^j, x\r \bigg)
\E \exp\bigg( i \sum_{j=1}^{n} \l
u^j, X(t^j)\r \bigg) d\overline{u}\ d\overline{t}
\end{split}
\end{equation}
and for all even integers $n \ge 2$,
\begin{equation} \label{Eq:moment2}
\begin{split}
\E\big[ \big(L(x, I) - L(y, I)\big)^n\big] =  &(2 \pi)^{- nd} \int_{I^n}
\int_{{\R}^{nd}} \prod_{j=1}^{n} \big[ e^{- i \l u^j, x\r} - e^{-
i \l u^j, y \r}\big]\\
& \qquad \qquad \qquad \times \E \exp \bigg( i \sum_{j=1}^n \l u^j, X(t^j)
\r\bigg) d\overline{u}\ d\overline{t},
\end{split}
\end{equation}
where $\overline{u} = ( u^1, \ldots, u^n),\ \overline{t} = (t^1,
\ldots, t^n),$ and each $u^j \in \R^d,\ t^j \in I \subseteq (0,
\infty)^N.$ In the coordinate notation we then write  $u^j = (u^j_1,
\ldots, u^j_d).$
Then upper bounds for \eqref{Eq:moment1} and \eqref{Eq:moment2} can
be derived by using Lemmas \ref{Lem:SLND}  and \ref{Lem:PG-th}. More
precisely, the proof of \eqref{Eq:m11} is the same as the proof of
Lemma 4.3, up to (4.28) in Xiao (2011), and the proof of \eqref{Eq:m20}
is the same as the proof of Lemma 4.7, up to (4.47) in Xiao (2011).
This finishes the proof of Lemma \ref{Lem:M1}.
\end{proof}


\subsection{Inverse image $X^{-1}(F)$}
\label{sec:Inverse}

In this section we  prove the following uniform result for the Hausdorff dimensions
of the inverse images $X^{-1}(F)$, where $F \subseteq \R^d$ are Borel
sets, of the harmonizable fractional stable sheet defined in (\ref{def:X}). For definition 
and basic properties (such as $\sigma$-stability and Frostman's lemma) of Hausdorff dimension,
we refer to Falconer (1990), or Khoshnevisan (2002).


\begin{theorem} \label{Th:IMdim2}
Let $X =\{X(t), t \in \R^N\}$ be an $(N, d)$ harmonizable fractional
stable sheet defined in (\ref{def:X}). We assume that $\a \in [1, 2]$
and $ \sum_{j=1}^N \frac 1 {H_j} > d$. Denote the random open set
\begin{equation}\label{Eq:O}
{\cal O} = \bigcup_{s, t \in \Q^N; \, 0 \prec s \prec t} \Big\{x \in \R^d:\
L(x, [s, t]) > 0 \Big\}.
\end{equation}
Then the following statements hold:
\begin{itemize}
\item[(i)]\, With probability 1,
\begin{equation}\label{Eq:IMdim2}
\begin{split}
\dim X^{-1}(F) &= \min_{1 \le k \le N} \bigg\{ \sum_{j=1}^k
\frac{H_k} {H_j} + N-k - H_k (d- \dim F)\bigg\}\\
&= \sum_{j=1}^k \frac{H_k} {H_j} + N-k - H_k (d- \dim F),  \
\hbox{ if } \sum_{j=1}^{k-1} \frac 1 {H_j} \le d- \dim F <
\sum_{j=1}^{k} \frac 1 {H_j}
\end{split}
\end{equation}
for all Borel sets $F \subseteq {\cal O}$. In the above we use the
convention $\displaystyle\sum_{j=1}^0H_j^{-1}=0$.
\item[(ii)]\, $\P\{ {\cal O}= \R^d\} > 0$.
\end{itemize}
\end{theorem}


In order to prove Theorem \ref{Th:IMdim2}, we will make use of the
following lemma where $Y$ can be quite general.
\begin{lemma}\label{Lem:uppbd}
Let $Y = \{Y(t), t \in \R^N\}$ be a function with values in $\R^d$.
We assume that for every bounded interval $T \subseteq \R^N$ and
every $\ep > 0$,
\begin{equation}\label{Eq:Y-Holder}
\big|Y(s) - Y(t) \big| \le \rho(s, t)^{1 - \ep}\qquad \forall s, t
\in T
\end{equation}
and $Y$ has a local time $L(x, T)$ which is bounded in the space variable $x$.
Then for every Borel set $F
\subseteq \R^d$,
\begin{equation}\label{Eq:updim1}
\dim Y^{-1}(F) \le \min_{1 \le k \le N} \bigg\{ \sum_{j=1}^k
\frac{H_k} {H_j} + N-k - H_k (d- \dim F)\bigg\}.
\end{equation}
\end{lemma}
\begin{proof}\ Our argument is reminiscent to the proof of a corresponding
result in Monrad and Pitt (1987). For any integer $n \ge 2$, we divide $T$
into sub-rectangles $R_{n,i}$ of side-lengths $2^{-n/H_j}$ ($j = 1, \ldots, N$).
Note that there are at most $c\, 2^{n Q}$ such rectangles (recall that $Q =
\sum_{j=1}^N \frac1 {H_j}$), and each $R_{n, i}$ is equivalent to a
ball of radius $2^{-n}$ in the metric $\rho$.

For any (open or closed) ball $B(y, 2^{-(1-\ep)n})$, let $N_\rho(y,
n, T)$ be the number of $R_{n, i}$'s such that $R_{n, i} \cap
X^{-1}(B(y, 2^{-(1-\ep)n})) \ne \varnothing$. For each such
rectangle $R_{n, i}$, the uniform H\"older condition
(\ref{Eq:Y-Holder}) implies that $X(R_{n, i}) \subseteq B(y,\,
2^{-(1-\ep)n+1})$. It follows from the occupation density formula
(\ref{Eq:occupation}) that
\begin{equation}\label{Eq:NT}
\begin{split}
N_\rho(y, n, T) \, 2^{-Q n} \le \int_{B(y, \, 2^{-(1-\ep)n +1})}
L(x, T)\, dx \le c\, 2^{- (1-\ep)d n }.
\end{split}
\end{equation}
This yields
\begin{equation}\label{Eq:NT2}
N_\rho(y, n, T) \le c_{11}\, 2^{(Q - (1-\ep)d) n}.
\end{equation}
Note that the constant  $c_{11}$ is independent of $y$ and $n$.

Now we prove \eqref{Eq:updim1} by using (\ref{Eq:NT2}) and a
covering argument. Note that it is sufficient to show that $\dim
\big(Y^{-1}(F)\cap T\big)$ is bounded by the right hand side of
\eqref{Eq:updim1}. The general result follows from the arbitrariness
of $T$ and the $\sigma$-stability of $\dim $.

Given any Borel set $F \subseteq \R^d$, we choose and fix an
arbitrary constant $\gamma > \dim F$. Then there exist a constant
$\delta >0$ and a sequence of balls $\{B(y_j, r_j), j \ge 1\}$ in
$\R^d$ (in Euclidean metric) such that $r_j \le \delta$ for all $j
\ge 1$,
\begin{equation}\label{Eq:Covering}
F \subseteq \bigcup_{j=1}^\infty B(y_j, r_j) \quad \hbox{ and }\quad
\sum_{j=1}^\infty (2r_j)^{\gamma} \le 1.
\end{equation}

For every integer $j \ge 1$, let $n_j$ be the integer such that
$2^{-(1-\ep) (n_j+1)} \le r_j < 2^{-(1 - \ep)n_j}$.  Since
$$
X^{-1}(F) \subseteq \bigcup_{j=1}^\infty X^{-1}\big(B(y_j,
r_j)\big),
$$
we obtain a covering of $X^{-1}(F) \cap T$ by a subsequence of
intervals $\{R_{n_j, i}\}$ (i.e., those satisfying $R_{n_j, i} \cap
X^{-1}\big(B(y_j, r_j)\big) \ne \varnothing$). For simplicity,
we set $N_j = N_\rho(y_j, n_j, T)$.

For every $k \in \{1, \ldots, N\}$, each rectangle $R_{n_j, i}$ can
be covered by at most $\prod_{\ell= k}^N \big(2^{n_j ( \frac 1 {H_k}
- \frac 1 {H_\ell})}$ $+1\big)$ \emph{cubes} of side-length
$2^{-n_j/H_k}$. In this way, we obtain a covering of $X^{-1}(F)\cap
T$ by cubes of side-length $2^{-n_j/H_k}$ which can be used to bound
the Hausdorff measure of $X^{-1}(F)\cap T$. Let
\[
\beta_k = \sum_{j=1}^k \frac{H_k} {H_j} + N-k - H_k (d- \gamma).
\]
It follows from \eqref{Eq:Covering} and \eqref{Eq:NT2} that
\begin{equation}\label{Eq:covering3}
\sum_{j=1}^\infty N_j \prod_{\ell= k}^N \Big(2^{n_j ( \frac 1 {H_k}
- \frac 1 {H_\ell})} + 1\Big)\, 2^{- \frac {n_j} {H_k} \,\beta_k}
\le c\, \sum_{j=1}^\infty 2^{- {n_j}\gamma} \le
c_{12}.
\end{equation}
This implies that ${\cal H}_{\beta_k}\big(X^{-1}(F) \cap T) \le
c_{12}$. Hence we have $\dim \big(X^{-1}(F)\cap T\big) \le
\beta_k$ for every $k \in \{1, \ldots, N\}$. Letting $\gamma
\downarrow \dim F$ yields the desired upper bound.
\end{proof}

We are ready to prove Theorem \ref{Th:IMdim2}.  

\begin{proof}{\bf of Theorem \ref{Th:IMdim2}}\ First we prove
Part (i). Since it is elementary to verify the second equality
in (\ref{Eq:IMdim2}), we will only prove the first equality in
(\ref{Eq:IMdim2}).

For any bounded interval $T\subseteq (0, \infty)^N$, by Lemmas
\ref{Lem:unimod} and \ref{Lem:JC}, $X$ satisfies a uniform H\"older
condition in the metric $\rho$ on $T$ of order $1-\ep$ and has a jointly
continuous (hence bounded) local time on $T$. Hence the upper bound
in (\ref{Eq:IMdim2}) follows from Lemma \ref{Lem:uppbd} and the 
$\sigma$-stability  of Hausdorff dimension.

For proving the lower bound in (\ref{Eq:IMdim2}), 
let $\gamma < \dim F$ be an arbitrary constant. Then by Frostman's lemma
[cf. Falconer (1990)] there exists a Borel probability 
measure $\nu$ on $\R^d$ supported by $F$ such that
\begin{equation}\label{Eq:nu1}
\nu(B(y, r))\le c_{13}\, r^\ga, \quad \forall \ y \in \R^d\
\hbox{ and }\ r > 0.
\end{equation}
We define a random Borel measure on $(0, \infty)^N$ by
\begin{equation}\label{Eq:nu2}
\mu (I) = \int_{\R^d} L(x, I)\, \nu(dx), \ \ \hbox{ for all } \  I \subset (0, \infty)^N.
\end{equation}
Since $\nu$ is supported by $F$ and $L(\cdot, I)$ vanishes outside
the closure $\overline {X(I)}$, we can see that $\mu$ is a finite measure
supported on $X^{-1}(F)$. Moreover, for every $F \subset {\cal O}$,
we have $\mu(X^{-1}(F)) > 0$.

Let $k \in \{1, \ldots, N\}$ be the integer satisfying
\begin{equation}
\sum_{j=1}^{k-1} \frac 1 {H_j} \le d- \dim F < \sum_{j=1}^{k} \frac
1 {H_j}.
\end{equation}
For any $r \in (0, 1)$, there is an integer $n\ge 0$ such that
\begin{equation}\label{Eq:nu3}
2^{-(n+1)/H_k} \le r < 2^{-n/H_k}.
\end{equation}
Hence, for every $t \in (0, \infty)^N$, the Euclidean ball $U(t, r)$ can be
covered by at most $N_n$ intervals $\{ I_i = [t_i, \, t_i + \l 2^{-n/H_\ell}\r ],\,
1 \le i \le N_n\}$ (each $I_i$ is a subset of the ball centered at $t_i$ in the
metric $\rho$ of radius $\sqrt{N}\,2^{-n}$),  where
\begin{equation}\label{Eq:nu4}
N_n \le c\, \prod_{j=k + 1}^N2^{- \frac{n} {H_k} + \frac{n} {H_j}} =
c\, 2^{n \big(\sum_{j=k+1}^N \frac1 {H_j}- \frac{N-k} {H_k}\big)}.
\end{equation}

Let us fix an $\omega \in \Omega$ such that the conclusions of
both Lemma  \ref{Lem:unimod}  and Theorem \ref{Th:Holder} hold.
Hence, by (\ref{Eq:Holder4}) and \eqref{eq:mod-cont}, we have that
a. s. for all $n$ large enough,
\begin{equation}\label{Eq:LT-i}
\max_{x \in \R^d} L(x, I_i) \le c_{12} 2^{-n (1 - \eps)
\big(\sum_{\ell = 1}^N\frac 1 {H_\ell} - d\big)}
\end{equation}
and
\begin{equation}\label{Eq:diam}
\max_{s, t \in I_i } \big|X(s) - X(t)\big| \le c\, 2^{-n (1-\eps)}
\end{equation}
for every $1 \le i \le N_n$.

It follows from (\ref{Eq:nu4}), (\ref{Eq:LT-i}), \eqref{Eq:nu1}, and
\eqref{Eq:diam} that for any $\eps > 0$ small
\begin{equation}\label{Eq:Rm-1}
\begin{split}
\mu (U(t, r)) &\le \sum_{i=1}^{N_n} \int_{\R^d} L(x, I_i )\, \nu(dx)\\
&\le c_{17}\, \sum_{i=1}^{N_n}\, \max_{x \in \R^d} L(x, I_i )\,
\big[ {\rm diam} X(I_i)\big]^\ga\\
&\le c\, 2^{-n \big(\sum_{j=1}^k  \frac1 {H_j}- \frac{N-k} {H_k} - d
+ \ga - \eps\big)} = c\, r^{\eta},
\end{split}
\end{equation}
where
\[
\eta = \sum_{j=1}^k  \frac {H_k} {H_j}- N +k  - H_k(d - \ga) - \eps',
\]
where $\eps' = c_{18}\, \eps$.
This, together with the Frostman's theorem [cf. Falconer (1990)],
implies that almost surely for every $F \subset {\cal O}$,
\[
\dim X^{-1}(F) \ge \sum_{j=1}^k  \frac {H_k} {H_j}- N +k - H_k(d - \ga)
- \eps'
\]
Letting $\eps \downarrow 0$ yields the desired lower bounds
for $\dim X^{-1}(F)$. This proves (\ref{Eq:IMdim2}).

In order to prove Part (ii), we first notice that, for any bounded
interval $I \subset (0, \infty)^N$ with non-empty interior, we have
$\E\big[L(0, I)\big] > 0$. This follows directly from
(\ref{Eq:moment1}), (\ref{Eq:ChF1}) and (\ref{eq:scale}). Hence
$\P\big\{L(0, I) > 0 \big\} > 0.$ Using the a.s. continuity of
$x \to L(x, I)$, we see that there are constants $r> 0$ and $\delta > 0$
such that $\P\big\{L(x, I) > 0 \hbox{ for all } \, x \in B(0, r)\big\} \ge \delta.$
This implies that
\begin{equation} \label{Eq:posip}
\P\big\{ B(0, r) \subset {\cal O} \big\} \ge \delta.
\end{equation}
On the other hand, by the scaling property (\ref{Eq:OSS}) and
the occupation formula (\ref{Eq:occupation}), one can verify that
the local times $L(x, I)$ have the following property:
For any integer $n \ge 2$,
\begin{equation}\label{Eq:LT-sim}
L(x, I ) \stackrel{d} = n^{Nd - Q} L(n^N x, n^E I),
\end{equation}
where $E$ is the $N\times N$ diagonal matrix
diag($n^{H_1^{-1}}, \ldots, n^{H_N^{-1}}$) and
$Q = \sum_{j=1}^N H_j^{-1}$. It follows from \eqref{Eq:posip}
and \eqref{Eq:LT-sim} that
$\P\big\{ B(0, n^N r) \subset {\cal O} \big\} \ge \delta.$
By Fatou's lemma, we have
\[
\P\big\{ B(0, n^N r) \subset {\cal O} \hbox{ for infinitely
many }\, n's\big\} \ge \delta.
\]
This proves $ \P\big\{  {\cal O} = \R^d  \big\} \ge \delta$.
\end{proof}

We remark that, even though in the present paper we have focused
on harmonizable fractional stable sheet $X$, the methods in this
section are applicable to stable random fields that have jointly
continuous local times such that their uniform moduli of continuity
satisfy (\ref{eq:mod-cont}) in Lemma \ref{Lem:unimod} and 
(\ref{Eq:Holder4}) in Proposition \ref{Th:Holder}, respectively. 
With some extra effort, one can verify that Part (i) of Theorem 
\ref{Th:IMdim2} holds for a large class of $(N, d)$ stable random 
fields $X$ defined as in (\ref{def:X}) such that $X_1$ is a 
harmonizable stable random field with stationary increments. 
For example, $X_1$ can be taken as an operator-scaling harmonizable 
stable random field constructed by Bierm\'e, Meerschaert 
and Scheffler (2007) (in this case, Part (ii) of Theorem 
\ref{Th:IMdim2} also holds) or a stable random field in Sections 
3.1 and 3.2 of Xiao (2011).

This paper also raises several interesting open questions.
We list some of them below.

\begin{itemize}
\item We believe that $\P\big\{{\cal O}= \R^d\big\} = 1$ for an $(N, d)$
harmonizable fractional stable sheet in Theorem \ref{Th:IMdim2}, but we
have not been able to prove this conjecture in general. When $\alpha = 2$,
this can be proved by applying Theorem 2 of Monrad and Pitt (1987) to the
stationary Gaussian random field obtained from a fractional Brownian sheet
via the Lamperti transform (cf. Genton et al (2007) for the latter). For
$\a \in (1, 2)$, the conjecture would be proven if one can establish a
Hewitt-Savage zero-one law for harmonizable fractional stable sheets.
See Orey and Pruitt (1973, p.141) and Tran (1976, p.32) for results for
the Brownian sheet.
\item By checking the proofs in Section 3, one can see that the results on
local times and uniform Hausdorff dimension of $X^{-1}(F)$ would hold for
harmonizable fractional $\alpha$-stable sheets with $\a \in (0, 1)$, if the
property of local nondeterminism of $Z^H$ had been proved when $\a \in (0, 1)$.
See Remark \ref{Re:LND}.
\item In light of Bierm\'e, Lacaux and Xiao (2009) and the present paper,
it is a natural question to determine the packing dimension (cf. Falconer (1990))
of the inverse image $X^{-1}(F)$ for $F \subset \R^d$ and to establish a
corresponding uniform packing dimension result. Solving these problems
would require   new results on the local times of harmonizable fractional
stable sheets. More  precisely, instead of Proposition \ref{Th:Holder} which
is concerned with uniform modulus of continuity of the local times, it will be
essential to study the liminf behavior of the local times. No results of this
kind have been proven for stable random fields.
\item  It would be interesting to establish a similar uniform dimension result
for linear fractional stable sheets (LFSS) considered in Ayache {\it et al} (2009).
The method for proving Theorem \ref{Th:IMdim2} breaks down for a LFSS
due to the fact that its uniform modulus of continuity is different from that
of $Z^H$ in Lemma \ref{Lem:unimod}, a new method will be needed.
\item In (\ref{def:X}), the real-valued random fields $X_1, \ldots, X_d$ are
assumed to be independent copies. This technical assumption brings convenience
for studying sample path properties of $X$ but it is obviously too restrictive.
Recently, Xiao and Li (2011), Kremer and Scheffler (2018), S\"onmez (2017, 2018)
have introduced and studied several classes of operator-self-similar $(N, d)$
stable random fields whose components are dependent in general. It would be of
interest to further investigate the fractal properties and local times of these
stable random fields.
\end{itemize}


\bibliographystyle{plain}
\begin{small}

\end{small}
\end{document}